\documentclass[10pt]{article}

\usepackage{amsmath,amssymb,amscd,amsthm,amsfonts}
\usepackage{graphicx,subfigure}
\usepackage{hyperref}
\usepackage{dsfont}
\usepackage{color}
\usepackage{float}
\usepackage{bm}

\graphicspath{{figures/}}

\theoremstyle{plain}
\newtheorem{thm}{Theorem}[section]

\newtheorem{lem}[thm]{Lemma}

\theoremstyle{definition}

\newtheorem{rem}[thm]{Remark}

\newtheorem{exmp}[thm]{Example}

\newcommand{\R}{\ensuremath{\mathbb{R}}}

\newcommand{\Z}{\ensuremath{\mathbb{Z}}}
\newcommand{\C}{\ensuremath{\mathbb{C}}}

\newcommand{\bp}{\ensuremath{\mathbb{P}}}

\newcommand{\N}{\ensuremath{\mathbb{N}}}
\newcommand{\cl}{\ensuremath{\mathcal{L}}}

\newcommand{\wq}{\ensuremath{\infty}}
\newcommand{\vp}{\ensuremath{\varphi}}

\newcommand{\wl}[1]{\ensuremath{\overline{#1}}}
\newcommand{\wt}[1]{\ensuremath{\widetilde{#1}}}
\newcommand{\p}{\ensuremath{\partial}}
\newcommand{\na}{\ensuremath{\nabla}}
\newcommand{\lc}{\ensuremath{\left(}}
\newcommand{\rc}{\ensuremath{\right)}}
\newcommand{\hk}{hyperk\"ahler}
\newcommand{\ms}{\ensuremath{\mathbb{S}}}

\newcommand{\id}{\ensuremath{\text{Id}}}

\newcommand{\ii}{{\ensuremath{\rm{i}}}}
\newcommand{\norm}[2]{{ \ensuremath{\left\|} #1 \ensuremath{\right\|}}_{#2}}
\newcommand{\he}{{\ensuremath{\bm{e}}}}
\newcommand{\df}{{\ensuremath{\mathrm{d}}}}

\title{Complex structures of the Gibbons-Hawking ansatz with infinite topological type}

\author{Wenxin He \and Bin Xu}
\date{}

\begin{document}

\maketitle
\begin{abstract}
In this paper, we study the complex structures of complete \hk\ four-manifolds of infinite topological type arising from the Gibbons-Hawking ansatz. We show that for almost all complex structures in the \hk\ family, the manifold is biholomorphic to a hypersurface in $\C^3$ defined by an explicit entire function. For the remaining complex structures, we further prove that the manifold is biholomorphic to the minimal resolution of a singular surface in $\C^3$ under certain conditions. Thus, we partially extend LeBrun's celebrated work \cite{lebrun1991complete} to the context of countably many punctures.
\end{abstract}


\section{Introduction}\label{section-introduction}
In their pioneering work, Gibbons and Hawking \cite{gibbons1979classification, GIBBONS1978430} introduced a construction of four-dimensional \hk\ metrics with a tri-Hamiltonian $S^1$-action, now known as the Gibbons-Hawking ansatz. Hitchin \cite{hitchin1979polygons}, through the approach of twistor space, developed a coherent picture of the \hk\ geometry behind these constructions. Kronheimer \cite{kronheimer1993construction} described four-dimensional asymptotically locally Euclidean spaces arising as finite-dimensional \hk\ quotients. Bielawski \cite{bielawski1999complete} analyzed complete \hk\ manifolds with local tri-Hamiltonian actions and general dimension, placing a wide range of examples into a unified framework. In these works, the relevant configurations of puncture points are finite, and the resulting manifolds are of finite topological type. On the other hand, Anderson, Kronheimer and LeBrun \cite{anderson1989complete} used Gibbons-Hawking metrics to construct complete Ricci-flat \hk\ manifolds with infinite topological type. Goto \cite{goto1994hyper} later constructed \hk\ manifolds with infinite topological type and general dimension from the viewpoint of infinite dimensional \hk\,quotients, closely related to Kronheimer's construction \cite{kronheimer1993construction} of \hk\ manifolds of type $A_k$. Subsequent work by Hattori \cite{hattori2011volume, hattori2014holomorphic} studied the volume growth and holomorphic symplectic structures of these type $A_\wq$ manifolds. Swann \cite{swann2016twists} provided a classification of complete four-dimensional \hk\ manifolds that admit a tri-Hamiltonian $S^1$-action, without imposing any assumptions on their topological type. Dancer and Swann \cite{dancer2017hypertoric, DANCER2019168} obtained further developments for general-dimensional \hk\ manifolds of infinite topological type with a tri-Hamiltonian torus action.

In this manuscript, inspired by the works  \cite{anderson1989complete, lebrun1991complete, kalafat2009hyperk, ChenChen+2019+259+284}, we investigate the complex structures of complete four-dimensional \hk\ manifolds $M$ of infinite topological type obtained from the Gibbons-Hawking ansatz. Our first result (Theorem~\ref{uniV}) gives an existence criterion for the harmonic functions in the Gibbons-Hawking ansatz for such manifolds. The corresponding existence statement was already proved in the work of Kalafat and Sawon \cite{kalafat2009hyperk} by an argument based on the maximum principle. We recover the criterion by a different method, using potential theory. Building on the method in LeBrun \cite{lebrun1991complete}, we show that for almost all complex structures in the \hk\ family on $M$, the underlying complex manifold is biholomorphic to a hypersurface in $\C^3$ defined by an explicit entire function (Theorem~\ref{almostc}). LeBrun's work treats the case where the circle action has finitely many fixed points and our description extends his picture to the countably many fixed points. For the remaining complex structures, which form a set of measure zero, we prove that under an additional finiteness condition, the complex manifold $M$ can be realized as the minimal resolution of a singular hypersurface in $\C^3$ (Theorem~\ref{vmr}). In this way we relate the complex structures arising from the Gibbons–Hawking ansatz with infinite topological type directly to the geometry of the puncture set. We also note that Hattori \cite{hattori2014holomorphic} proved a similar classification theorem for four-dimensional hyperkähler manifolds of type $A_\infty$.

\subsection{Background of Gibbons-Hawking ansatz}\label{backgrounds}
In this subsection we recall the Gibbons-Hawking construction and fix the particular \hk\ manifold of infinite topological type that will be studied throughout the paper.

Let $U\subset\R^3$ be an open subset with coordinates $(x,y,z)$ and let $\pi_U:M_U\to U$ be a principal $S^1$-bundle. $V$ is a positive harmonic function defined on $U$ such that $\frac{1}{2\pi}\star \df V$ represents the Chern class of this bundle, where $\star$ is the Hodge star operator with respect to the Euclidean metric $g_{\text{Euc}}$ on $\R^3$. Then there is a connection $1$-form $\omega$ whose curvature is $\star {\df}V$. The Gibbons-Hawking ansatz then produces a \hk\ metric on $M_U$:
\begin{align*}
  g=\frac{1}{V}\omega^2+V\pi_U^*g_{\text{Euc}}.
\end{align*}
The corresponding K\"ahler form and complex structures are
\begin{align*}
  &\Omega_x=\df x\wedge \omega+V {\df}y\wedge {\df}z,\ J_x:{\df}x\mapsto V^{-1}\omega,\ {\df}y\mapsto {\df}z\\
  &\Omega_y={\df}y\wedge \omega+V {\df}z\wedge {\df}x,\ J_y:{\df}y\mapsto V^{-1}\omega,\ {\df}z\mapsto {\df}x\\
  &\Omega_z={\df}z\wedge \omega+V {\df}x\wedge {\df}y,\ J_z:{\df}z\mapsto V^{-1}\omega,\ {\df}x\mapsto {\df}y,
\end{align*}
and these satisfy the quaternionic relations
\begin{align*}
  J_x^2=J_y^2=J_z^2=-\id,\ J_xJ_y=-J_yJ_x=J_z.
\end{align*}
Thus $g$ is a \hk\ metric with a tri-Hamiltonian $S^1$-action generated by the circle fibres.

We now specialize to the class of Gibbons-Hawking metrics of infinite topological type introduced by Anderson-Kronheimer-LeBrun \cite{anderson1989complete}. Consider a divergent sequence of distinct points $p_j=(x_j,y_j,z_j)\in\R^3,\ j\in\N^+$ and set $A=\{p_j:j\in\N^+\},\ U=\R^3\setminus A$. Let $\pi_0:M_0\to \R^3\setminus A$ be a principal $S^1$-bundle whose first Chern class has value $-1$ on a small sphere around each puncture $p_j$. More precisely, if $S_j^2$ is a sphere of sufficiently small radius enclosing only the point $p_j$, then the first Chern class $c_1$ of $\pi_0$ satisfies
\begin{align*}
  \int_{S^2_j}c_1=-1\text{  for every }j\in\N^+.
\end{align*}

Choose $r_j>0$ such that $r_j<\min_{k\ne j}||p_k-p_j||$ so that $B_{r_j}(p_j)\setminus\{p_j\}$ contains no other punctures. Then $\pi_0^{-1}(B_{r_j}(p_j)\setminus\{p_j\})$ is diffeomorphic to a punctured 4-ball $\hat{B_j}-\{0\}\subset\R^4$. We compactify $M_0$ by attaching back the centres of these $4$-balls and define
\begin{align*}
  M:=M_0\cup\ \bigcup_{j=1}^\wq \hat{B_j}/\sim
\end{align*}
where $\sim$ represents the identification between $\hat{B_j}-\{0\}$ and $\pi_0^{-1}(B_{r_j}(p_j)\setminus\{p_j\})$. Then $\pi_0:M_0\to\R^3$ can be extended to a smooth map $\pi:M\to\R^3$ for which $\{p_j\}$ are the fixed points of the $S^1$-action. 

Assume that the puncture points satisfy the summability condition
\begin{align}\label{convergence}
  \sum_{j=2}^\infty\frac{1}{\norm{p_1-p_j}{}}<\wq.
\end{align}
Then the function
\begin{align}\label{funcV}
  V(p):=\frac12\sum_{j=1}^\wq\frac{1}{\norm{p-p_j}{}}
\end{align}
is a well-defined smooth harmonic function on $\R^3\setminus\{p_j\}$. One checks that the closed $2$-form $\frac{1}{2\pi}\star {\df}V$ is a representative element of the Chern class $(-1,-1,\cdots)$ of the principal $S^1$-bundle. Hence there is a connection on $\pi_0$ with curvature $\star {\df}V$. Let $\omega$ denote the connection form of this connection. The Gibbons-Hawking metric
\begin{align*}
  g=\frac{1}{V}\omega^2+V\pi_0^{*}g_{\text{Euc}}
\end{align*}
on $M_0$ extends smoothly across the added $S^1$-fixed points $\pi^{-1}(p_j)\subset M$, and Anderson-Kronheimer-LeBrun \cite{anderson1989complete} showed that the resulting metric $g$ on $M$ is complete and \hk, while $M$ has infinite topological type. For simplicity, we call this harmonic function $V$ the {\it Gibbons-Hawking potential}.

\subsection{Main results}\label{main}
In this section, we present the main results derived from the \hk\ manifolds of infinite topological type constructed in the previous subsection.

Throughout, 
\begin{itemize}
  \item Let $\{p_j\}_{j=1}^\wq\subset\R^3$ be a closed, discrete and countable set of points, with $p_j=(x_j,y_j,z_j)$. We take $A:=\{p_j:j\in\N^+\}$ and $U:=\R^3\setminus A$. Then
\[
H^2(U,\Z)=\prod_{j=1}^\wq\Z.
\]
We write a cohomology class $\he\in H^2(U,\Z)$ as $\he=(e_1,e_2,\cdots)$.
  \item Let $\ms^2\subset \R^3$ be the unit sphere 
    \begin{align*}
        \ms^2:=\{(x,y,z)\in \R^3: x^2+y^2+z^2=1\}.
    \end{align*}
    We denote the standard measure on $\ms^2$ by $m$. For a direction $v\in \ms^2$, we define the complex structure in the \hk\ family of $M$ by
    \begin{align*}
        J_v:=v_1J_x+v_2J_y+v_3J_z.
    \end{align*}
    \item Let $\Pi_v$ be the orthogonal projection from $\R^3$ onto the plane $v^\perp$ orthogonal to $v$. Define $a_j(v):=\Pi_v(p_j)$ and let $\{b_1(v),b_2(v),\cdots\}$ be the set of distinct values of $\{a_j(v)\}$.
\end{itemize}

Our first result is an existence criterion for the positive harmonic functions which arise as Gibbons-Hawking potential on $U$.
\begin{thm}\label{uniV}
  For $U=\R^3\setminus A$ where $A=\{p_j:j\in\N^+\}\subset \R^3$ is a closed, discrete and countable set, a nontrivial cohomology class 
  \begin{align*}
    \bm{e}=(e_1,e_2,\cdots)\in H^2(U,\Z)
  \end{align*}
  can be represented in the form of $[\star {\df}V_\he]$ where $V_\he$ is some positive harmonic function on $U$, if and only if $e_j\le0$ for all $j\in\N^+$ and there exists a point $x\in U$ such that
  \begin{align*}
  \sum_{j\ge1}\frac{|e_j|}{|x-p_j|}<\wq.
  \end{align*}
\end{thm} 

\begin{rem}
In the concrete Gibbons-Hawking set-up in Section~\ref{backgrounds}, the principal $S^1$-bundle $\pi_0:M_0\to\R^3\setminus\{p_j\}_{j=1}^\wq$ is chosen so that its first Chern class takes the value $-1$ on a small sphere enclosing each puncture $p_j$. Equivalently, the corresponding cohomology class $e=(e_1,e_2,\cdots)$ is fixed by the condition $e_j=-1$ for all $j\in\N^+$. Then we obtain a positive harmonic function $V_\he$ on $U=\R^3\setminus\{p_j\}_{j=1}^\wq$ of the form
\begin{align*}
  V_\he(x)=\sum_{j=1}^\wq\frac{1}{4\pi|x-p_j|}
\end{align*}
so that $[\star {\df}V_\he]=e$. This is exactly the harmonic function appearing in \eqref{funcV} since $V=2\pi V_\he$. In particular, the requirement $e_j=-1$ for all $j$ ensures that the resulting \hk\ metric extends smoothly across the added $S^1$-fixed points over $\{p_j\}_{j=1}^\wq$, producing the complete \hk\ manifold $M$ of infinite topological type described in Section~\ref{backgrounds}.
\end{rem}

\begin{thm}\label{almostc}
    Assume that the punctures $\{p_j\}$ satisfy
    \begin{align*}
        \sum_{j=2}^\wq \frac{1}{||p_1-p_j||}<\infty.
    \end{align*}
    Then there exists a subset $B$ of $\mathbb{S}^2$ with $m(\ms^2\setminus B)=0$ such that for every $v\in B$, $(M,J_v)$ is biholomorphic to a hypersurface in $\C^3$ defined by 
    \begin{align*}
        u_1u_2=P_v(u_3)
    \end{align*}
    where $P_v$ is an entire function determined by $\{p_j\}$ and $v$.
\end{thm}

When $v\notin B$, $M$ can no longer be described directly as a hypersurface in $\C^3$. Nevertheless, under an additional finiteness assumption on the multiplicities of the projection points, $(M,J_v)$ can be realized as the minimal resolution of a singular hypersurface.

\begin{thm}\label{vmr}
    Assume that the punctures $\{p_j\}$ satisfy
    \begin{align*}
        \sum_{j=2}^\wq \frac{1}{||p_1-p_j||}<\infty.
    \end{align*}
    If $v\in \ms^2\setminus B$ is a direction satisfies the following conditions:
    \begin{itemize}
        \item [(1).] The set $\{b_k(v)\}_{k=1}^\wq$ has no accumulation point on $\C$;
        \item [(2).] For each $k\ge1$, the following multiplicities are finite:
        \begin{align*}
          &m_0(v):=\#\{j\in\mathbb{N}^+:a_j(v)=0\}<\wq,\\
          &m_k(v):=\#\{j\in\mathbb{N^+}:a_j(v)=b_k(v)\}<\infty.
        \end{align*}
    \end{itemize}
    Then $(M,J_v)$ is biholomorphic to the minimal resolution of a surface in $\C^3$ defined by 
    \begin{align*}
        u_1u_2=P_v(u_3)
    \end{align*}
where $P_v$ is an entire function determined by $\{p_j\}$ and $v$.
In particular, the surface has at most countably many orbifold singularities
of type $A_{m_k(v)-1}$ for each $m_{k(v)}$ greater than one.
\end{thm}

\subsection{Outline}
The paper is structured as follows. In Section~\ref{s2}, we investigate the Riesz measure of the harmonic function $V$ and give the proof of Theorem~\ref{uniV}. In Section~\ref{s3}, 
we first analyse the case for a fixed complex structure and prove our key Lemma~\ref{mbi}, which gives the biholomorphism between $M$ and a hypersurface in $\C^3$. Combining this with a measure argument, we establish Theorem~\ref{almostc} for almost all complex structures in the \hk\ family of $M$. Finally, in Section~\ref{s4}, we focus on the remaining complex structures and we prove that $M$ is the minimal resolution of a hypersurface under the additional finiteness condition, as stated in Theorem~\ref{vmr}.

\section{Existence criteria for harmonic functions via Riesz measures}\label{s2}
In this section we justify the choice of Gibbons-Hawking potential $V$ and give the proof of Theorem~\ref{uniV}. The strategy is to apply the uniqueness of the Riesz measure associated to a subharmonic function and to compute this measure explicitly in terms of the puncture configuration.

Let $U=\R^3\setminus A$ where $A=\{p_j:j\in\N^+\}\subset \R^3$ is a closed, discrete and countable set. Then
\[
H^2(U,\Z)=\prod_{j=1}^\wq\Z.
\]
We write a cohomology class $\he\in H^2(U,\Z)$ as $\he=(e_1,e_2,\cdots)$.

When $e_j\le0\ \forall\ j\in \N^+$ and
\begin{align*}
  \sum_{j\ge1}\frac{|e_j|}{|x-p_j|}<\wq,
\end{align*}
we can define
\begin{align*}
  V_\he(x)=\sum_{j=1}^\infty\frac{|e_j|}{4\pi|x-p_j|},\ x\in U.
\end{align*}
The series converges locally uniformly on $U$, and hence $V_\he$ is a well-defined harmonic function on $U$. To see this, set
\begin{align*}
    V_\he^{k}(x)=\sum_{j=1}^{k}\frac{|e_j|}{4\pi|x-p_j|}.
\end{align*}
Each $V_\he^k$ is a positive harmonic function on $U$, and the sequence $\{V_\he^k\}_{k\ge1}$ is monotonically increasing. We invoke the following version of Harnack's theorem (see \cite[Theorem 1.20]{MR460672}).
\begin{thm}
    [Harnack's Theorem]
    For $\{u^k(x)\}$ a sequence of monotonically increasing harmonic functions on a connected open subset $U$ of $\R^3$, if there exists a point $q\in U$ such that the sequence $\{u^k(q)\}$ is bounded above, then there exists a harmonic function $u(x)$ on $U$ such that
    \begin{align*}
       \lim_{k\to\wq} u^k(x)=u(x)
    \end{align*}
    uniformly on compact subsets of $U$.
\end{thm}
In our case, choosing $x\in U$ with
\begin{align*}
  \sum_{j\ge1}\frac{|e_j|}{|x-p_j|}<\wq,
\end{align*}
ensures that $\{V_\he^k(x)\}$ is uniformly bounded. By Harnack's theorem, the sequence $\{V_\he^k\}_k$ converges locally uniformly on $U$ to a harmonic function, which we denote by $V_\he$. Thus $V_\he$ is positive and harmonic on $U$.

Moreover, for $j\in\N^+$,
\begin{align*}
    \int_{\partial B(p_j,\epsilon_j)}\star {\df}V_\he=e_j
\end{align*}
for $\epsilon_j$ small enough, which implies that $[\star {\df}V_\he]=\he$.

Conversely, we show that the conditions in Theorem~\ref{uniV} are necessary. Suppose $V_\he$ is a positive harmonic function such that
\begin{align*}
  [\star {\df}V_\he]=\he=(e_1,e_2,\cdots)\in H^2(U,\Z),
\end{align*}
and assume, without loss of generality, that $e_1>0$. We appeal to the B$\hat{\text{o}}$cher's Theorem \cite[Theorem 3.9]{axler2013harmonic}.
\begin{thm}
    [B$\hat{\text{o}}$cher's Theorem]
    If $V$ is a positive and harmonic function on $B(0,1)\setminus\{0\}\subset\R^m$ with $m>2$, then there exists a function $v$ harmonic on $B(0,1)$ and a constant $b\ge0$ such that
    \begin{align}
        V(x)=v(x)+b|x|^{2-m}
    \end{align}
    for all $x\in B(0,1)\setminus\{0\}$.
\end{thm}

Therefore, if $V_\he$ is a positive harmonic function on $U$ such that $[\star {\df}V_\he]=\he$, there exists a real harmonic function $V_1$ on $B(p_1,\epsilon)$ and a real constant $b_1\ge0$ such that
\begin{equation*}
    V_\he=V_1+b_1\cdot\frac{1}{4\pi |x-p_1|},\ \ x\in B(p_1,\epsilon)\setminus \{p_1\}
\end{equation*}
for $\epsilon$ small enough. Furthermore,
\begin{align*}
  e_1=\int_{\p B(p_1,{\frac{\epsilon}{2}})}\star {\df}V_\he=\int_{\p B(p_1,{\frac{\epsilon}{2}})}\star {\df}V_1+\int_{\p B(p_1,{\frac{\epsilon}{2}})}\star {\df}\lc b_1\cdot\frac{1}{4\pi|x-p_1|} \rc=-b_1.
\end{align*}
If $e_1>0$, this gives $b_1<0$, which contradicts the requirement $b_1\ge0$ in B$\hat{\text{o}}$cher's Theorem. Hence no such positive $V_\he$ can exist if some $e_j>0$.

To extract the additional summability condition, we now use the language of Riesz measures. We recall two standard results from the theory of subharmonic functions \cite[Theorem 3.9, Theorem 3.20]{MR460672}.

\begin{thm}
    If $u(x)$  is a subharmonic function on a connected open subset $U\subset\R^m$ with $m\ge 3$ and not identically $-\wq$, then there exists a unique Radon measure $\mu$ in $U$ satisfying that for any compact subset $E$ of $U$, there is a harmonic function $h(x)$ defined on the interior of $E$ and for $x$ where $h$ is defined,
    \begin{align}
        u(x)=\int_EK(x-\xi){\df}\mu+h(x)
    \end{align}
    where
    \begin{align}
        K(x)=-|x|^{2-m}.
    \end{align}
    The measure $\mu$ is the Riesz measure of $u$.
\end{thm}

\begin{thm}\label{subharmonic}
    Given any Radon measure $\mu$ defined on $\R^m$ where $m\ge3$, if $n(t)=\mu(\wl{B(0,t)})$, then $\mu$ is the Riesz measure of some subharmonic function $u(x)$ which is bounded from above on $\R^m$ if and only if
    \begin{align}\label{vfb}
        \int_1^\wq\frac{n(t)}{t^{m-1}}{\df}t<\wq.
    \end{align}
    Moreover, when this condition holds, the unique subharmonic function corresponding to $\mu$ is given by
    \begin{align}
        u(x)=C-\int_{\R^m}\frac{1}{|x-\xi|^{m-2}}{\df}\mu.
    \end{align}
    where $C$ is the least upper bound of $u(x)$.
\end{thm}
\begin{proof}
  [Proof of Theorem~\ref{uniV}]

Assume $V_\he$ is a positive harmonic function on $U$ with $[\star {\df}V_\he]=\he$. We extend it to a subharmonic function on $\R^3$ by setting
\begin{align}\label{eforu}
    -V_\he(p_i):=\limsup_{x\to p_i}(-V_\he)(x).
\end{align}
In this way, the function $-V_\he$ is upper semi-continuous on $\R^3$, satisfies $-\wq\le -V_\he(x)<\wq$, and obeys the mean-value inequality: for any point $x_0\in\R^3$ there exists a sufficiently small $r>0$ such that
\begin{align*}
    (-V_\he)(x_0)\le\frac{1}{4\pi r^2}\int_{S(x_0,r)}(-V_\he)(x){\df}\sigma(x),
\end{align*} 
where ${\df}\sigma(x)$ denotes surface area on the sphere $S(x_0,r)$. Define $u:=-V_\he$. Then $u$ is a subharmonic function on $\R^3$. Let $\mu$ be the Riesz measure of $u$ and set $n(t)=\mu(\wl{B(0,t)})$. We shall calculate $n(t)$.

We first compute $\mu$ locally near each puncture point. For $\epsilon_j>0$ small and $x\in B(p_j,\epsilon_j)\setminus \{p_j\}$, $-u(x)=V_\he(x)$ is a positive harmonic function. Then we can apply B$\hat{\text{o}}$cher's Theorem,
\begin{align*}
  V_\he(x)=\frac{|e_j|}{4\pi|x-p_j|}+\wt{V_j}(x)
\end{align*}
where $\wt{V_j}$ is a harmonic function on $B(p_j,\epsilon_j)$. Thus $u(x)$ and $-\frac{|e_j|}{4\pi|x-p_j|}$ determine the same Riesz measure on $B(p_j,\epsilon_j)$. In particular, for any $\vp\in C_0^\wq(B(p_j,\epsilon_j))$,
\begin{align*}
  \int_{B(p_j,\epsilon_j)}\vp {\df}\mu=-\int_{B(p_j,\epsilon_j)}\frac{|e_j|}{4\pi|x-p_j|}\cdot\Delta\vp {\df}x=|e_j|\cdot\vp(p_j),
\end{align*}
so $\mu$ has an atom of mass $|e_j|$ at $p_j$.

Fix $t>0$ and set
\begin{align*}
  I_t=\{p_j:|p_j|\le t\}=A\cap\overline{B(0,t)}.
\end{align*}
Then $I_t$ is a finite set. For $\epsilon>0$ sufficiently small, we can arrange that $I_t=A\cap B(0,t+\epsilon)$. Arguing as above, there exists harmonic function $v_t(x)$ on $B(0,t+\epsilon)$ such that 
\begin{align*}
  V_\he(x)=\sum_{j:p_j\in I_t} \frac{|e_j|}{4\pi|x-p_j|}+v_t(x),\ \ x\in B(0,t+\epsilon)\setminus I_t
\end{align*}
and 
\begin{align*}
  n(t)=\sum_{j:p_j\in I_t}|e_j|.
\end{align*}
Based on the argument above,
\begin{align*}
  n(t)=\mu(\wl{B(0,t)})=\sum_{j=1}^\wq |e_j|\cdot \chi_{\{|p_j|\le t\}}(t),
\end{align*}
where 
\begin{align*}
  \chi_{\{|p_j|\le t\}}(t)=
  \left\{
  \begin{aligned}
    &1,&\text{ if }|p_j|\le t,\\
    &0,&\text{ if }|p_j|>t.
  \end{aligned}
  \right.
\end{align*}

To apply Theorem~\ref{subharmonic}, we now evaluate the integral~\eqref{vfb}
\begin{align*}
  \int_1^\wq\frac{n(t)}{t^{m-1}}{\df}t
\end{align*}
in the case $m=3$. Without loss of generality, we may reorder the points so that
\begin{align*}
  0<|p_1|\le|p_2|\le|p_3|\le\cdots
\end{align*}
which preserves the set $A$. Assume that there exists some index $i$ such that $|p_i|\le1$. Let $N$ be the index satisfying 
\begin{align*}
  |p_1|\le\cdots\le|p_N|\le 1<|p_{N+1}|\le\cdots.
\end{align*}
Then
\begin{align*}
        \int_1^\wq\frac{n(t)}{t^{2}}{\df}t
        &=\int_{1}^{|p_{N+1}|}\frac{1}{t^2}\lc \sum_{i=1}^N |e_i|\rc {\df}t
          +\sum_{i\ge N+1}^\wq\int_{|p_{i}|}^{|p_{i+1}|}\frac{1}{t^2}\lc \sum_{j=1}^{i}|e_j| \rc {\df}t\\
        &=\lc1-\frac{1}{|p_{N+1}|}\rc\lc\sum_{i=1}^{N}|e_i| \rc+\sum_{i\ge N+1}\lc\frac{1}{|p_i|}-\frac{1}{|p_{i+1}|}\rc\lc \sum_{j=1}^{i}|e_j| \rc\\
        &=\sum_{i=1}^{N}|e_i| +\sum_{i\ge N+1}\frac{|e_i|}{|p_i|}.
\end{align*}
If instead $|p_j|>1$ for all $j\in\mathbb{N}^+$, then the same computation yields
\begin{align*}
  \int_1^\wq\frac{n(t)}{t^{m-1}}{\df}t
  &=\sum_{j=1}^{\infty}\frac{|e_j|}{|p_j|}.
\end{align*}
Finally, by a simple application of the triangle inequality, the convergence of $\sum_j\frac{|e_j|}{|p_j|}$ is equivalent to the existence of a point $x\in U$ such that
\begin{align*}
  \sum_{j=1}^\wq\frac{|e_j|}{|x-p_j|}<\infty.
\end{align*}
Thus Theorem~\ref{subharmonic} shows that the summability condition in Theorem~\ref{uniV} is also necessary, completing the proof.

\end{proof}

\section{Hypersurface realizations in the \hk\ family}\label{s3}
In this section we describe the complex structures in the \hk\ family of $M$ which arise from certain directions. We begin by fixing a direction along the $x$-axis, and show that $(M,J_x)$ is biholomorphic to a smooth hypersurface in $\C^3$, which is recorded in Lemma~\ref{mbi}. We then prove that for almost every direction on $\ms^2$ and the corresponding complex structure $J_v$, the same description applies, thereby obtaining Theorem~\ref{almostc}. Our approach builds on the work of LeBrun~\cite{lebrun1991complete}, but is adapted to the case where the $S^1$-action has countably many fixed points.

Let $H:\R^3\to\C$ be the projection $H(x,y,z):=y+\ii z$ and $a_j:=H(p_j)=y_j+\ii z_j$. We say that the configuration $\{p_j\}_{j=1}^\wq$ is {\it generic} if the projection
\begin{align*}
  H:\{p_j\}_{j=1}^\wq\to\C
\end{align*}
is injective, that is, $\{a_j\}_{j=1}^\wq$ are pairwise distinct. Note that the projection $H$ here is exactly $\Pi_x$ introduced in Section~\ref{main}.

In the sequel we shall mainly work with the complex structure $J_x$ and for simplicity, we write $J:=J_x$. For later use it is convenient to describe it in slightly more concrete coordinates. Writing $\omega={\df}t+\theta$ for some $1$-form $\theta$ such that ${\df}\theta=\star {\df}V$, where $t$ is the angular coordinate along the $S^1$-fibres, the complex structure $J$ can be characterized by
\begin{align}\label{J}
  {\df}x\mapsto V^{-1}({\df}t+\theta),\ {\df}y\mapsto {\df}z.
\end{align}

Our key Lemma is the following description for $(M,J)$ under the generic hypothesis on the puncture points. For this we recall the standard Weierstrass factors
\begin{align*}
  E_0(u)=1-u,\ \ E_j(u)=(1-u)\exp\lc u+\frac{u^2}{2}+\cdots+\frac{u^j}{j} \rc.
\end{align*}

\begin{lem}\label{mbi}
    Suppose that the puncture points $\{p_j\}$ satisfy
\begin{align*}
  \sum_{j=2}^\wq\frac{1}{\|p_1-p_j\|}<\wq,
\end{align*}
and that the configuration $\{p_j\}_{j=1}^\wq$ is generic. If, in addition, the set $\{a_j\}_{j=1}^\wq$ has no accumulation point on $\C$, then
\begin{align*}
  P(u)=u^{\delta}\prod_{a_j\ne 0}E_{j}\lc u/a_j \rc
\end{align*}
where 
\begin{align*}
\delta=
    \left\{
    \begin{aligned}
        0\ \ &\text{ if }0\notin\{a_j\},\\
        1\ \  &\text{ if }0\in\{a_j\},
    \end{aligned}
    \right.
\end{align*}
is an entire function. Moreover $(M,J)$ is biholomorphic to the hypersurface in $\C^3$ given by
\begin{align*}
  u_1u_2=P(u_3).
\end{align*}
\end{lem}

\subsection{A holomorphic $\C^*$ action on $M$}
Let $\frac{\p}{\p t}$ denote the vector field generating the $S^1$-action along the fibres of $\pi:M\to\R^3$. Since the \hk\ structure is $S^1$-invariant, we have $\cl_{\frac{\p}{\p t}}J=0$, so $\frac{\p}{\p t}$ is a real holomorphic vector field with respect to $J$. Its $(1,0)$-part is 
\begin{align*}
  \lc \frac{\p}{\p t} \rc^{1,0} =\frac12\lc \frac{\p}{\p t}-\ii J\frac{\p}{\p t} \rc.
\end{align*}
We now define a holomorphic vector field
\begin{align*}
  \xi:=-\ii\lc \frac{\p}{\p t} \rc^{1,0}=-\frac {\ii}{2}\lc \frac{\p}{\p t}-\ii J\frac{\p}{\p t}  \rc=\frac 12 \lc V^{-1}\hat{\frac{\p}{\p x}}-\ii\frac{\p}{\p t}\rc
\end{align*}
where $\hat{\frac{\p}{\p x}}$ is the horizontal lift of $\frac{\p}{\p x}$ with respect to the connection $\omega$. The last equality follows from the explicit description of $J$ in \eqref{J}.

Note that $V^{-1}\hat{\frac{\p}{\p x}}$ is the horizontal lift of the vector field $V^{-1}\frac{\p}{\p x}$ on $\R^3$. Hence the completeness of flow generated by $V^{-1}\hat{\frac{\p}{\p x}}$ is equivalent to that of flow generated by $V^{-1}\frac{\p}{\p x}$ on $\R^3$. Consider an integral curve $(x(s),y,z)$ of $V^{-1}\frac{\p}{\p x}$. Along such a curve we have
\begin{align*}
  \frac{{\df}x(s)}{{\df}s}=V^{-1}(x,y,z),
\end{align*}
and we have
\begin{align*}
  \int_{x(s_1)}^{x(s_2)}V(x,y,z){\df}x=\int_{s_1}^{s_2}{\df}s=s_2-s_1.
\end{align*}
Since $V$ is positive and smooth on $U$, these integral curves extend for all $s\in\R$, and the flow of $V^{-1}\frac{\p}{\p x}$ on $\R^3$ is complete. It follows that the real and imaginary parts of $\xi$ are complete vector fields, and hence $\xi$ generates a holomorphic $\C$-action with complete complex flow. We denote this flow by
\begin{align*}
  \Phi_\lambda:M\to M,\ \lambda\in\C.
\end{align*}
Writing $\lambda=t+\ii s$, the flow may be decomposed as
  \begin{align*}
    \Phi_{t+is}(p)=\phi_t\circ\psi_s(p)=\psi_s\circ\phi_t(p)
  \end{align*} 
  where the real flows $\phi_t$ and $\psi_s$ are defined by
  \begin{align*}
    &\frac{\p\phi_t}{\p t}=V^{-1}\hat{\frac{\p}{\p x}}=-J\frac{\p}{\p t},\\
    &\frac{\p\psi_s}{\p s}=JV^{-1}\hat{\frac{\p}{\p x}}=\frac{\p}{\p t}.
  \end{align*}
In particular, $\psi_s$ coincides with the original $S^1$-action and $\psi_s=\psi_{s+2\pi}$. Thus $\Phi_{t+\ii s}=\Phi_{t+\ii s+2\pi \ii}$. The $\C$-action on $M$ descends via the exponential map to a holomorphic $\C^*$-action on $M$. Concretely,
\begin{align*}
  \C^*\times M\to M,\ (w,p)\mapsto\Phi_{\log w}(p).
\end{align*}

Next we describe the orbits of this $\C^*$-action. The holomorphic function
\begin{align*}
  H:M&\to\C,\\
 (x,y,z,t) &\mapsto y+\ii z
\end{align*}
is the natural extension of the projection $H:\R^3\to\C$ introduced in Section~\ref{main}. For $u\in\C$, let
\begin{align*}
      L_u:=\{(x,y,z)\in\R^3:y+\ii z=u\}
\end{align*}
denote a straight line in $\R^3$ parallel to $x$-axis and set $C_u:=\pi^{-1}(L_u)\subset M$. Then we have:
\begin{itemize}
    \item If $u\ne y_j+ \ii z_j$, $C_u$ is a single orbit.
    \item If $u=y_j+ \ii z_j$ for some $j$, $C_u$ is the union of three orbits: the fixed point $\{q_j:=\pi^{-1}(p_j)\}$, together with two open orbits $\pi^{-1}(\{(x,y_j,z_j):x>x_j\})$ and $\pi^{-1}(\{(x,y_j,z_j):x<x_j\})$.
\end{itemize}

Consider
\begin{align*}
  \Psi:=-\ii({\df}y\wedge\omega+V{\df}z\wedge {\df}x)+({\df}z\wedge \omega+V{\df}x\wedge {\df}y)
\end{align*}
which is a parallel $(2,0)$-form preserved by $\xi$. The holomorphic function $H:M\to \C,\ q=(x,y,z,t)\mapsto y+\ii z$ satisfies
\begin{align*}
  {\df}H=\iota_{\xi}\Psi.
\end{align*}

$M_0=M\setminus\{\pi^{-1}(\{p_j\})\}$ is the open subset of $M$ where the $S^1$-action is free. We can cover $M_0$  by two open subsets:
\begin{align*}
  &M^+:=\pi^{-1}\lc\{u\ne y_j+ \ii z_j\text{ for all }j\}\cup\bigcup_{j}\{u=y_j+\ii z_j\text{ and}\ x>x_j\}\rc\\
  &M^-:=\pi^{-1}\lc\{u\ne y_j+\ii z_j\text{ for all }j\}\cup\bigcup_{j}\{u=y_j+\ii z_j\text{ and}\ x<x_j\}\rc.
\end{align*}
Then $H:M^{\pm}\to\C$ are holomorphic principal $\C^*$-bundles. $\C$ is a Stein manifold. Since $\C^*$ deformation-retracts to $S^1$, we have
\begin{align*}
  [\C,K(\Z,2)]\cong H^2(\C,Z)=0.
\end{align*}
Hence the principal $\C^*$-bundles over $\C$ are all topologically trivial. Then by the Oka-Grauert Principle, \cite[Theorem 8.2.1]{forstnerivc2011stein}, $M^{\pm}$ are both biholomorphic to $\C\times\C^*$ as holomorphic principal $\C^*$-bundles over $\C$.

\begin{thm}
[Oka-Grauert Principle] For Stein manifold $X$ and every complex Lie group $G$, the holomorphic and the topological classes of principal $G$-bundles over $X$ are in one-to-one correspondence.
\end{thm}

Hence there exist biholomorphisms
\begin{align*}
  H^\pm:M^\pm&\to \C\times\C^*\\
  q&\mapsto (u,v)_\pm
\end{align*}
Under these identifications, the vector field $\xi$ is mapped to $v\frac{\p}{\p v}$ and the $2$-form $\Psi$ is mapped to $\frac{{\df}v\wedge {\df}u}{v}.$ The induced $\C^*$-action is therefore
\begin{align*}
  \lambda\cdot(u,v)=(u,\lambda v),\ \lambda\in\C^*.
\end{align*}

\subsection{The equivalence relation}
Now $M_0$ is biholomorphic to
\begin{align*}
  (\C\times\C^*)\amalg(\C\times\C^*)/\sim
\end{align*}
where $\sim$ denotes the equivalence 
\begin{align*}
  (u,v)_+\sim(u,f(u)v)_-
\end{align*}
for some holomorphic function $f:\C\setminus\{y_j+\ii z_j\}\to \C^*$. 
\begin{lem}\label{biholo}
For holomorphic $g:\C\to\C^*$, the two gluings defined by $f$ and $fg$ give biholomorphic complex manifolds. 
\end{lem}
\begin{proof}
Define holomorphic changes of trivialization
\begin{align*}
  \psi^+:(u,v)_+&\mapsto (u,v)_+,\\
  \psi^-:(u,v)_-&\mapsto (u,g(u)v)_-.
\end{align*}
Then the equivalence relation would be
\begin{align*}
  \psi^-\circ H^-\circ(H^+)^{-1}\circ(\psi^+)^{-1}=(u,f(u)g(u)v).
\end{align*}
Therefore, the complex structure can be determined only by behaviors of $f$ near the punctures.
\end{proof}

\begin{lem}\label{equi}
  On the overlap $M^+\cap M^-$, the equivalence relation can be given by
\begin{align*}
  (u,v)_+\sim\lc u,\frac{v}{P(u)} \rc_-
\end{align*}
where
\begin{align*}
  P(u)=u^{\delta}\prod_{a_j\ne 0}E_{j}\lc u/a_j \rc.
\end{align*}
and
\begin{align*}
\delta=
    \left\{
    \begin{aligned}
        0\ \ &\text{ if }0\notin\{a_j\},\\
        1\ \  &\text{ if }0\in\{a_j\}.
    \end{aligned}
    \right.
\end{align*}
\end{lem}

We shall apply the following complex Morse Lemma, \cite[Proposition 3.15]{ebeling2007functions}
\begin{lem}
  Let $M$ be a $2$-dimensional complex manifold, $H:M\to\C$ a holomorphic function and $q$ a nondegenerate critical point of $H$. Then there is a local coordinate system $(w_1,w_2)$ in a neighbourhood $V$ of $q$ such that on $V$ we have
  \begin{align*}
    H(w_1,w_2)=H(q)+w_1w_2,\ \ w_1(q)=w_2(q)=0.
  \end{align*}
\end{lem}

\begin{proof}
  [Proof of Lemma~\ref{equi}]
Now $H:M\to\C$ has nondegenerate critical points at the isolated points $q_j=\pi^{-1}(p_j)$ for $p_j=(x_j,y_j,z_j)$. Indeed, we have
\begin{align*}
  {\df}H_{p_j}=\iota_\xi\Psi|_{p_j}=0
\end{align*}
and the Hessian of $H$ at $q_j$ can be expressed as
\begin{align*}
  \na {\df}H(v,w)=\lc \na\iota_{\xi}\Psi(v,w) \rc=\lc \na_v\Psi(\xi,w) \rc-\Psi(\xi,\na_vw)=\Psi(\na_v\xi,w)
\end{align*}
Note that the $\C^*$-action generated by $\xi$ has isolated fixed points $\{p_j\}$, the linear map $\na\xi|_{q_j}:T_{q_j}M\to T_{q_j}M$ is invertible. As $\Psi$ is nondegenerate, it follows that the Hessian of $H$ is non-degenerate at each $q_j$.
By the Morse Lemma each $q_j$ has a neighbourhood $U_j$ with a holomorphic chart on which we could find complex coordinates $(w^j_1,w^j_2)$ such that
\begin{align*}
  &\vp^j=(w^j_1,w^j_2):U_j\to V_j\subset \C^2,\ \vp^j(q_j)=0,\\ 
  &H\circ(\vp^j)^{-1}(w^j_1,w^j_2)=y_j+\ii z_j+w^j_1w^j_2=a_j+w^j_1w^j_2.
\end{align*}
When no confusion can arise, we drop the index $j$ for brevity. On $(U_j,\vp^j)$, we have
\begin{align*}
  \Psi=\psi(w_1,w_2){\df}w_1\wedge {\df}w_2
\end{align*}
for some nonvanishing holomorphic function $\psi(w_1,w_2)$. Then $\xi=a\frac{\p}{\p w_1}+b\frac{\p}{\p w_2}$ can be determined by
\begin{align*}
  \iota_\xi\Psi={\df}H=w_1{\df}w_2+w_2{\df}w_1=\iota_\xi (\psi {\df}w_1\wedge {\df}w_2),
\end{align*}
which implies that
\begin{align*}
  \xi=\psi^{-1}(w_1,w_2)\lc w_1\frac{\p}{\p w_1}-w_2\frac{\p}{\p w_2} \rc.
\end{align*}

Since $\xi$ generates $\C^*$-action,
\begin{align*}
  \int_0^{2\pi}\psi(e^{\ii\theta}w_1,e^{-\ii\theta}w_2){\df}\theta=2\pi.
\end{align*}
The complex flow of $\xi$ with respect to $\lambda=s+it$
\begin{align*}
  &\frac{{\df}w_1}{{\df}\lambda}=\psi^{-1}w_1\\
  &\frac{{\df}w_2}{{\df}\lambda}=-\psi^{-1}w_2
\end{align*}
Integrate $\lambda$ along a closed circle around $0\in\C$,
\begin{align*}
  \oint {\df}\lambda=2\pi \ii=\oint\psi\frac{{\df}w_1}{w_1}=\ii\int_0^{2\pi} \psi(e^{\ii\theta}w_1,e^{-\ii\theta}w_2){\df}\theta.
\end{align*}
Consider the expansion of $\psi$
\begin{align*}
  \psi=\sum_{k,l=0}^\wq a_{kl}\cdot w_1^kw_2^l,\ a_{00}=1,\ a_{jj}=0,\ \ \forall\ j>0.
\end{align*}
One can solve the equation
\begin{align*}
  \lc w_1\frac{\p}{\p w_1}-w_2\frac{\p}{\p w_2} \rc\alpha=\psi-1
\end{align*}
by
\begin{align*}
  \alpha=\sum_{k\ne l}\frac{a_{kl}}{k-l}w_1^kw_2^l.
\end{align*}
Now set 
\begin{align*}
  s^j_1:=e^{\alpha}w_1^j,\ s^j_2:=e^{-\alpha}w_2^j.
\end{align*}
Then $(s^j_1,s^j_2)$ gives a chart around $q_j$ satisfying
\begin{align*}
  H=s_1^js_2^j+a_j,\ a_j=y_j+\ii z_j.
\end{align*}
and we have
\begin{align*}
  \xi=s_1\frac{\p}{\p s_1}-s_2\frac{\p}{\p s_2}.
\end{align*}
The flow generated by $\xi$ is given by
\begin{align*}
  \Phi_\lambda(s_1,s_2)=(e^{\lambda}s_1,e^{-\lambda}s_2),\ \lambda\in\C.
\end{align*}

Claim:
Near the critical point $q_j$ where $u=H\ne y_j+ \ii z_j$, the equivalence relation is 
\begin{align*}
  f_j(u)=\frac{\epsilon^2}{u-(y_j+\ii z_j)}=\frac{\epsilon^2}{u-a_j}
\end{align*}
where $\epsilon$ is a non-vanishing complex constant.
Under coordinate transformation, the equivalence function can always be represented by $f_j(u)g_j(u)$ where $g_j$ is a holomorphic non-vanishing function on $U_j$.

Indeed, we have
\begin{align*}
  U_j\cap M^+=\{s_1^js_2^j\ne0\}\sqcup(\{s_1^j=0, s_2^j\ne0\}\cap M^+)\sqcup(\{s_1^j\ne0,s_2^j=0\}\cap M^+),\\
  U_j\cap M^-=\{s_1^js_2^j\ne0\}\sqcup(\{s_1^j=0, s_2^j\ne0\}\cap M^-)\sqcup(\{s_1^j\ne0,s_2^j=0\}\cap M^-).
\end{align*}
Since $U_j\cap M^+\setminus \{s_1^js_2^j\ne0\}$ is $U_j\cap M^+\setminus\{u\ne a_j\}=U_j\cap\{u=a_j,\ x>x_j\}$ which contains only one orbit, it must be one of the two cases, either $\{s_1^j=0, s_2^j\ne0\}\cap M^+$ or $\{s_1^j\ne0,s_2^j=0\}\cap M^+$.

Without loss of generality, we may assume the former one. $(U_j\cap M^+)\cup (U_j\cap M^-)$ has been mapped onto $V_j\setminus\{0\}\subset\C^2$, then we can exactly pick standard section $\sigma_+$ for $U_j\cap M^+$ and $\sigma_-$ for $U_j\cap M^-$:
\begin{align*}
  &\sigma_+=(s_1(u),s_2(u))=(\epsilon,\frac{u-a_j}{\epsilon}),\\
  &\sigma_-=(s_1(u),s_2(u))=(\frac{u-a_j}{\epsilon},\epsilon),
\end{align*}
for some fixed non-zero constant $\epsilon\in \C$. The local trivializations for $H:M^{\pm}\to\C$ restricted on $U_j$ can be written as
\begin{align*}
  h_j^+:V_j^+ &\to M^+\cap U_j \\
       (u,v^+)& \mapsto (s_1,s_2)=\Phi_{\log v^+}(\epsilon,\frac{u-a_j}{\epsilon})=(\epsilon v^+,\frac{u-a_j}{\epsilon v^+}).
\end{align*}

\begin{align*}
  h_j^-:V_j^- &\to M^-\cap U_j \\
       (u,v^-)& \mapsto (s_1,s_2)=\Phi_{\log v^-}(\frac{u-a_j}{\epsilon},\epsilon)=(\frac{u-a_j}{\epsilon}v^-,\frac{\epsilon}{v^-}).
\end{align*}
where $V_j^+\subset (H(U_j)-\{a_j\})\times \C^*$, $V_j^-\subset (H(U_j)-\{a_j\})\times \C^*$ are two suitable open subsets.

And the equivalence relation is determined by
\begin{align*}
  (h_j^-)^{-1}\circ h_j^+(u,v^+)=(u,\frac{\epsilon^2}{u-a_j}v^+).
\end{align*}

Recall that under the generic assumption we have a countable set $\{a_j\}$ of pairwise distinct points of $\C$ without accumulation, and we define
\begin{align*}
  P(u)=u^{\delta}\prod_{a_j\ne 0}E_{j}\lc u/a_j \rc
\end{align*}
where
\begin{align*}
\delta=
    \left\{
    \begin{aligned}
        0\ \ &\text{ if }0\notin\{a_j\},\\
        1\ \  &\text{ if }0\in\{a_j\},
    \end{aligned}
    \right.
\end{align*}
and where $E_j$ are the standard Weierstrass factors. Then $P(u)$ is an entire function with simple zeros exactly at the points $a_j$.

If the equivalence relation of $M_0$ is $f$, then $f$ must have exactly simple poles at $a_j$ and no other singularities. Now $f$ and $\frac{1}{P}$ have the same zeros and poles of same multiplicities at $\{a_j\}$, then the quotient 
\begin{align*}
  g(u):=f(u)P(u)
\end{align*}
extends to an entire holomorphic function $g:\C\to\C^*$. From Lemma~\ref{biholo} and based on the discussions above, we have coordinates $(u,v_\pm)$ on $M^\pm$ and the equivalence relation is given by $\frac{1}{P(u)}$.

\end{proof}

We now complete the proof of our key Lemma.

\begin{proof}
  [Proof of Lemma~\ref{mbi}]
Consider the hypersurface
\begin{align*}
  \Sigma:=\{(u_1,u_2,u_3)\in\C^3:u_1u_2=P(u_3)\}.
\end{align*}
Suppose $\sigma:=\{(0,0,y_j+\ii z_j):j\in\N^+\}\subset\C^3$ and define
\begin{align*}
  \chi:M_0&\to\Sigma\setminus\sigma,\\
  (u,v_+)&\mapsto(u_1,u_2,u_3)=(\frac{P(u)}{v_+},v_+,u),\\
  (u,v_-)&\mapsto(u_1,u_2,u_3)=(\frac{1}{v_-},P(u)v_-,u).
\end{align*} 
Then $\chi(M^+)=\Sigma^+=\{u_2\ne0\}\cap \Sigma\setminus\sigma$, $\chi(M^-)=\Sigma^-=\{u_1\ne0\}\cap \Sigma\setminus\sigma$. $\chi$ is well-defined and holomorphic on $M_0$. 

Conversely, on $\Sigma\setminus\sigma$ we define
\begin{align*}
  (u_1,u_2,u_3)\mapsto
  \left\{
    \begin{aligned}
      &(u_3,u_2)_+\ &\text{ if }u_2\ne0,\\
      &\left(u_3,\frac{1}{u_1}\right)_-\ &\text{ if }u_1\ne0,
    \end{aligned}
  \right.
\end{align*}
and this map gives a holomorphic inverse to $\chi$. Hence $\chi$ is a biholomorphism from $M_0$ to $\Sigma\setminus\sigma$ and locally maps $U_j\setminus \{q_j\}$ to the intersection of $\Sigma\setminus\sigma$ and a neighbourhood of $(0,0,y_j+\ii z_j)$ in $\Sigma$.

By the following Hartogs's theorem \cite[Theorem 2.3.2]{hormander1973introduction}, $\chi$ extends uniquely to a biholomorphism from $M$ to $\Sigma$. This proves Lemma~\ref{mbi}
\end{proof}
\begin{thm}
  [Hartogs's extension theorem]
  For an open set $U\subset\C^n,\ n\ge2$, any holomorphic function on $U\setminus\{q\}$ which is a set punctured at one point $q\in U$ can be extended to a holomorphic function on $U$.
\end{thm}

\begin{rem}
Lemma~\ref{mbi} recovers, as a special case, the example of Anderson-Kronheimer-LeBrun \cite{anderson1989complete} in which $\{p_j\}_j=\{(0,0,z_j)\in\R^3\}_j$ with $\sum_{j=1}^\wq\frac{1}{|z_j|}<\wq$.

In this situation, $a_j=\ii z_j$ and the Weierstrass product reduces to
\begin{align*}
  P(u)=\prod_{j=1}^\wq\lc 1-\frac{u_3}{\ii z_j} \rc.
\end{align*}

So $(M,J)$ is biholomorphic to
\begin{align*}
  u_1u_2=\prod_{j=1}^\wq\lc 1-\frac{u_3}{z_j} \rc.
\end{align*} 
\end{rem}

\subsection{Transition to other directions in $\ms^2$}

So far we have worked with the complex structure $J=J_x:{\df}x\mapsto V^{-1}({\df} t+\theta),\ {\df}y\mapsto {\df}z$, for which the $\R^3$-projections of the $\C^*$-orbits are parallel to the $x$-axis. Recall that we have
\begin{align*}
  &J_y:{\df}y\mapsto V^{-1}({\df}t+\theta),\ {\df}z\mapsto {\df}x,\\
  &J_z:{\df}z\mapsto V^{-1}({\df}t+\theta),\ {\df}x\mapsto {\df}y.
\end{align*}
Together, these three complex structures satisfy the quaternionic identities
\begin{align*}
  J_x^2=J_y^2=J_z^2=-\id,\ J_xJ_y=-J_yJ_x=J_z.
\end{align*}
Then for each direction $v=(v_1,v_2,v_3)\in \ms^2\subset \R^3$, we obtain a compatible complex structure 
\begin{align*}
  J_v:=v_1J_x+v_2J_y+v_3J_z
\end{align*}
for which the $\R^3$-projections of the $\C^*$-orbits are parallel to $v$. The projection to the plane $v^\perp$ plays the role of $H$ in this setting.

For $v\in \ms^2$, recall $a_j(v)=\Pi_v(p_j)$ and we define the entire function
\begin{align*}
   P_v(u):=u^{\delta(v)}\prod_{a_j(v)\ne0}E_j\lc\frac{u}{a_j(v)}\rc.
\end{align*}
where 
\begin{align*}
\delta(v)=
    \left\{
    \begin{aligned}
        0\ \ &\text{ if }0\notin\{a_j(v)\},\\
        1\ \  &\text{ if }0\in\{a_j(v)\}.
    \end{aligned}
    \right.
\end{align*}

We now show that for almost every directions $v\in\ms^2$, the assumptions of Lemma~\ref{mbi} are satisfied and the previous construction applies. 

\begin{proof}[Proof of Theorem~\ref{almostc}]
  Without loss of generality, we may assume $p_j\ne(0,0,0)$ for all $j$, and set $v_j:=\frac{p_j}{||p_j||}\in \ms^2$. For $n\in\N$, define $A_{j,n}\subset \ms^2\subset\R^3$ by
  \begin{align*}
    A_{j,n}:=\left\{v\in \ms^2:\theta(v,v_j)\le \frac{n}{||p_j||}\right\},
  \end{align*}
  where $\theta(v,v_j)$ denotes the angle between $v$ and $v_j$. Geometrically, $A_{j,n}$ is a spherical cap of opening angle $\frac{n}{||p_j||}$. Its measure satisfies
  \begin{align*}
    m(A_{j,n})=2\pi\lc 1-\cos \frac{n}{||p_j||}\rc\le \pi\lc\frac{n}{||p_j||}\rc^2.
  \end{align*}
  For any fixed $n$ we have
  \begin{align*}
    \sum_{j=1}^\infty m(A_{j,n})\le(\pi n^2)\sum_{j=1}^\infty\frac{1}{||p_j||^2}
  \end{align*}
  Since
  \begin{align*}
    \sum_{j\ge2}\frac{1}{||p_1-p_j||}<\wq,
  \end{align*}
  we have
  \begin{align*}
    \sum_{j=1}^\infty\frac{1}{||p_j||^2}<\wq.
  \end{align*}

By the Borel-Cantelli lemma, for $A_n:=\{v\in \ms^2:\ v\in A_{j,n}\text{ for infinitely many }j\}=\limsup_{j\to\infty}A_{j,n}$,
  \begin{align*}
    m(A_n)=0.
  \end{align*}

Then the probability that $\Pi_v(p_j)$ has some accumulation point is
  \begin{align*}
    &\frac{m\lc\left\{ v\in \ms^2:\text{There exists }n<\infty,\ v\in \ms^2,\text{ such that }\#\{j:||\Pi_v(p_j)||\le n\}=\infty\right\}\rc}{m(\ms^2)}\\
    =&P\lc\text{There exists }n<\infty,\ v\in \ms^2,\text{ such that }\#\{j:||\Pi_v(p_j)||\le n\}=\infty\rc\\
    =&P\lc\text{There exists }n<\infty,\ v\in \ms^2,\text{ such that }\#\left\{j:\sin\theta\le\frac{n}{||p_j||}\right\}=\infty\rc\\
    \le &P\lc \text{There exists }n<\infty,\ v\in \ms^2,\text{ such that }\#\left\{j:\theta\le\frac{n}{||p_j||}\right\}=\infty \rc\\
    =&0.
  \end{align*}
  
Secondly, the set 
  \begin{align*}
    B:=\left\{ v\in \ms^2: v=\pm\frac{p_i-p_j}{||p_i-p_j||}\text{ for some }i,j\right\}
  \end{align*}
  is a countable set. Hence for every $v\in \ms^2\setminus B$, $\Pi_v$ is injective on $\{p_j\}$. 

Putting these facts together, we conclude that for almost every $v\in \ms^2$, the points $\{a_j(v)\}$ are pairwise distinct and have no accumulation point in $\C$. For such $v$, the conclusion of Lemma~\ref{mbi} holds with $H$ replaced by $\Pi_v$ and $J$ by $J_v$. Therefore, for such $v$, $(M,J_v)$ is biholomorphic to the hypersurface
\begin{align*}
  \{(u_1,u_2,u_3)\in \C^3:u_1u_2=P_v(u_3)\}\subset\C^3.
\end{align*}
This is exactly the statement of Theorem~\ref{almostc}.
\end{proof}

\section{Non-injective projection and the minimal resolution}\label{s4}
While the previous discussion focused on the case where $\Pi_v$ is injective from $\{p_j\}_j$ to their images, in this section we study the case where several punctures project to the same point in $\C$. In this case $M$ can no longer be described directly as a hypersurface in $\C^3$. Instead, under certain conditions we show that $M$ is realized as the minimal resolution of a singular hypersurface. As in Section~\ref{s3}, we first treat the case of a fixed projection direction along the $x$-axis, assuming that only finitely many punctures lie over each point of the image. Then, by replacing this fixed direction with an arbitrary direction $v\in\ms^2$ and working with the corresponding complex structure $J_v$, we obtain Theorem~\ref{vmr}.

Before turning to the general countable configuration, we first recall a simple model case in which exactly two fixed points lie in the same image of $H$. This example, due to Chen-Chen \cite{ChenChen+2019+259+284}, already illustrates how surface singularities of type $A_1$ arise in this setting and how they are resolved. It will serve as a concrete motivation for the general statement of Theorem~\ref{vmr}.

\begin{exmp}[Chen-Chen \cite{ChenChen+2019+259+284}]
Without loss of generality, assume that there are only two fixed points under the $S^1$ action: $q_1=\pi^{-1}(p_1),\ q_2=\pi^{-1}(p_2)$ where $p_1=(0,0,0)$ and $p_2=(1,0,0)$. In this situation, Chen-Chen \cite{ChenChen+2019+259+284} has shown that the resulting \hk\ manifold $M$ is biholomorphic to the minimal resolution of the hypersurface in $\C^3$ defined by $u_1u_2=u_3^2$.

Define 
\begin{align*}
  M_1&:=\pi^{-1}\lc\{u\ne0\}\cup\{u=0\text{ and }x<0\}\rc,\\
  M_2&:=\pi^{-1}\lc\{u\ne0\}\cup\{u=0\text{ and }0<x<1\}\rc,\\
  M_3&:=\pi^{-1}\lc\{u\ne0\}\cup\{u=0\text{ and }x>1\}\rc.
\end{align*}
Then $M_1\cup M_2\cup M_3$ is an open cover of $M_0:=M\setminus\{\pi^{-1}(\{p_1,p_2\})\}$ and each $M_j,\ j\in\{1,2,3\}$ is a holomorphic principal $\C^*$-bundle under $H:M_j\to\C$. $M_0$ is biholomorphic to
\begin{align*}
  (\C\times \C^*)_{M_1}\amalg (\C\times\C^*)_{M_2}\amalg (\C\times\C^*)_{M_3}/\sim
\end{align*}
where the equivalence relation $\sim$ is generated by
\begin{align*}
  (u,v)_{M_1}\sim(u,\frac{v}{u})_{M_2}\sim(u,\frac{v}{u^2})_{M_3}.
\end{align*}

Let $S$ be the hypersurface in $\C^3$ defined by the following equation
\begin{align*}
  u_1u_2=u_3^2,
\end{align*}
and let $Bl_0(S)$ denote the blowing-up of $S$ at point $(0,0,0)\in\C^3$:
\begin{align*}
  Bl_0(S)&=\{((u_1,u_2,u_3),[U_1,U_2,U_3]):u_jU_k=u_kU_j\ (1\le j<k\le3), U_1U_2=U_3^2\}\\
  &\subset S\times\bp^2.
\end{align*}
Remove the two points of the exceptional divisor over the origin by setting
\begin{align*}
  Bl_0(S)':=Bl_0(S)\setminus\{\lc(0,0,0),[1:0:0]\rc,\lc(0,0,0),[0:1:0]\rc\}.
\end{align*}
On each $M_j$ we choose local coordinates $(u,v_j)$ and define a map $\chi_0:M_0\to Bl_0(S)'$ by
\begin{align*}
  \chi_0|_{M_1}(u,v_1)&=\lc\lc \frac{u^2}{v_1},v_1,u\rc,\left[\lc \frac{u}{v_1} \rc^2:1:\frac{u}{v_1}\right] \rc,\\
  \chi_0|_{M_2}(u,v_2)&=\lc\lc\frac{u}{v_2},uv_2,u\rc,\left[\frac{1}{v_2}:v_2:1\right]\rc,\\
  \chi_0|_{M_3}(u,v_3)&=\lc\lc \frac{1}{v_3},u^2v_3,u \rc,\left[1:\lc uv_3\rc^2:u v_3\right]\rc.
\end{align*}
These expressions are compatible on overlaps, so $\chi_0$ is well-defined and holomorphic.

Conversely, consider the open subsets
\begin{align*}
  \Sigma_1:=\{u_2\ne0\}\cap Bl_0(S)',\\
  \Sigma_2:=\{U_3\ne0\}\cap Bl_0(S)',\\
  \Sigma_3:=\{u_1\ne0\}\cap Bl_0(S)'.
\end{align*}
Define
\begin{align*}
  &\Theta_0|_{\Sigma_1}\lc\lc u_1,u_2,u_3 \rc,\left[U_1:U_2:U_3\right]\rc=(u_3,u_2)_{M_1},\\
  &\Theta_0|_{\Sigma_2}\lc\lc u_1,u_2,u_3 \rc,\left[U_1:U_2:U_3\right]\rc=(u_3,\frac{U_2}{U_3})_{M_2},\\
  &\Theta_0|_{\Sigma_3}\lc\lc u_1,u_2,u_3 \rc,\left[U_1:U_2:U_3\right]\rc=(u_3,\frac{1}{u_1})_{M_3}.
\end{align*}
Again, one checks that $\Theta_0$ is well-defined and holomorphic, and that $\Theta_0$ and $\chi_0$ are inverse to each other. Hence we have the biholomorphism between $M_0$ and $Bl_0(S)'$. By the Hartogs's theorem, this biholomorphism extends across the two missing points to a biholomorphism $M\cong Bl_0(S)$. So in the two-centre case $(M,J)$ is exactly the minimal resolution of the singular surface $S$.
\end{exmp}

The two-centre model already exhibits the essential phenomenon that will reappear in the general situation. In the general case, several punctures may lie over each value of the projection. We now turn to this general configuration for a fixed projection direction.

Recall that $H:\R^3\to \C$, $H(x,y,z)=y+\ii z$, is the projection along $x$-axis and $a_j=H(p_j)=y_j+\ii z_j$. Let $\{b_1,b_2,\cdots\}$ be the set of distinct values of $\{a_j\}$. 
We shall state the main Lemma for the fixed direction along the $x$-axis and the corresponding complex structure $J=J_x$.

\begin{lem}\label{mr}
For countable distinct points $p_j\in\R^3,\ j\in\mathbb{N^+}$, suppose that
\begin{align*}
  \sum_{j=2}^\wq\frac{1}{\|p_1-p_j\|}<\wq.
\end{align*}
Assume that
\begin{itemize}
  \item [(1).] the set $\{b_k\}_{k=1}^\wq$ has no accumulation point on $\C$;
  \item [(2).] for each $k\ge1$, the following multiplicities are finite:
  \begin{align*}
    &m_0:=\#\{j\in\mathbb{N}^+:a_j=0\}<\wq,\\
    &m_k:=\#\{j\in\mathbb{N^+}:a_j=b_k\}<\infty.
  \end{align*}
\end{itemize}
Relabel the points $\{p_j\}_j$ as $\{p_{k,l}\}_{k,l}$ so that $H^{-1}(b_k)=\{p_{k,l}:1\le l\le m_k\}$ for all $k\in\N^+$. Using suitable Weierstrass factors, we define an entire function
\begin{align*}
  P(u):=u^{m_0}\prod_{k:b_k\ne0}\lc E_{\ell_k}\lc \frac{u}{b_k} \rc \rc^{m_k}
\end{align*}
where the positive integers $\ell_k$ are properly chosen so that the infinite product converges uniformly on every compact subset of $\C$.

Let $S\subset\C^3$ be the surface defined by
\begin{align*}
  u_1u_2=P(u_3).
\end{align*}
Then $(M,J)$ is biholomorphic to the minimal resolution of $S$. The singularity type is $(A_{m_1-1},\cdots,A_{m_k-1},\cdots)$ for each $k$ with $m_k\ge2$.
\end{lem}

\begin{proof}

For $\{q_j\}_{j=1}^\wq=\{\pi^{-1}(p_j)\}_{j=1}^\wq$, the charts that we take to cover $M_0:=M\setminus\{q_j\}_{j=1}^\wq$ are
\begin{align*}
  M^-:&=\pi^{-1}\lc \left\{u\ne b_k\text{ for all }k\right\} \cup\bigcup_{k}\left\{u=b_k\text{ and }x<x_{k,1}\right\} \rc\\
  M^+:&=\pi^{-1}\lc \left\{u\ne b_k\text{ for all }k\right\} \cup\bigcup_{k}\left\{u=b_k\text{ and }x>x_{k,m_k}\right\} \rc\\
  M_{k,l}:&=\pi^{-1}\lc \{u\ne b_k\text{ for all }k\}\cup\{u=b_k\text{ and }x_{k,l}<x<x_{k,l+1}\} \rc
\end{align*}
for $1\le l\le m_k-1$ and $k\ge1$. Each chart is mapped by $H$ to an open subset of $\C$ and hence the base manifold is Stein. The $\C^*$-action makes each restriction $H:M_\alpha\to H(M_\alpha)$ a trivial holomorphic principal $\C^*$-bundle.

We introduce coordinates $(u,v_\alpha)$ on each $M_\alpha$ and describe the gluing in the form
\begin{align*}
  (u,v_\alpha)\sim \lc u,f_{\alpha\beta}(u)v_\beta\rc.
\end{align*}
The equivalence relation $f_{\alpha\beta}$ are expressed using integer-valued functions $J_r(\alpha)$:
\begin{align*}
  J_r(\alpha):=
  \left\{
   \begin{aligned}
    &0,\ &M_\alpha=M^-,\\
    &l,&M_\alpha=M_{r,l},\\
    &m_r,&M_\alpha=M^+,\\
    &0,&r\ne k\text{ on }M_\alpha=M_{k,l}.
   \end{aligned}
  \right.
\end{align*}
Then
\begin{align*}
  f_{\alpha\beta}(u)=u^{J_0(\alpha)-J_0(\beta)}\prod_{k:b_k\ne0}\lc E_{\ell_k}\lc\frac{u}{b_k}\rc\rc^{J_{k}(\alpha)-J_k(\beta)}.
\end{align*}

For the surface $S\subset C^3$ defined by $u_1u_2=P(u)$, consider the map $\chi_0:M_0\to S\setminus\{0,0,b_k\}_k$:
\begin{align*}
  \chi_0|_{M^-}(u,v_-)&=\lc \frac{P(u)}{v_-},v_-,u  \rc,\\
  \chi_0|_{M^+}(u,v_+)&=\lc \frac{1}{v^+},P(u)v_+,u \rc,\\
  \chi_0|_{M_{k,l}}(u,v_{k,l})&=\lc \frac{P(u)(u-b_k)^{-l}}{v_{k,l}},(u-b_k)^lv_{k,l},u \rc.
\end{align*}
These three families of formulas agree on overlaps, so $\chi_0$ is a well-defined, holomorphic map. It extends holomorphically over the deleted points by setting $\chi(q_{k,l})=(0,0,b_k)$, $q_{k,l}=\pi^{-1}(p_{k,l})$ and therefore defines a holomorphic map $\chi:M\to S$. 

For each $k$ with $m_k\ge2$ and $1\le l\le m_k-1$, consider the curve
\begin{align*}
  E_{k,l}:&=\wl{\{(u,v_{k,l})\in M_{k,l}:u=b_k\}},\ 1\le l\le m_k-1.
\end{align*}
By regarding $(u,v_{k,l})$ as $p_{k,l}$ when $v_{k,l}=0$ and as $p_{k,l+1}$ when $v_{k,l}=\wq$, we see that  each $E_{k,l}$ is biholomorphic to $\bp^1$ and their intersections satisfy
\begin{align*}
  E_{k_1,l_1}\cap E_{k_2,l_2}=
  \left\{
  \begin{aligned}
    &\{(u,v_{k,l_1})=(b_k,\wq)\}=\{(u,v_{k,l_2})=(b_k,0)\},&k_1=k_2=k,\ l_2=l_1+1,\\
    &\emptyset, &k_1\ne k_2\text{ or }|l_1-l_2|\ge2.
  \end{aligned}
  \right.
\end{align*}
For $k$ with $m_k\ge2$, set
\begin{align*}
  E_k:&=\bigcup_{l=1}^{m_k-1}E_{k,l}=\chi^{-1}(0,0,b_k)\subset M,\\
  E:&=\bigcup_{k:m_k\ge2} E_k.
\end{align*}
Then
\begin{align*}
  \chi^{-1}(0,0,b_k)=E_k,
\end{align*}
and the restriction $\chi|_{M\setminus E}:M\setminus E\to S\setminus\{(0,0,b_k)\}_{k}$ is a biholomorphism. Since $\chi$ is proper, it follows that $\chi:M\to S$ is a resolution of singularities of $S$. 

To compute the self-intersection of the exceptional curves, we apply the adjunction formula
\begin{align*}
    2g(E_{k,l})-2=\lc K_M+E_{k,l} \rc\cdot E_{k,l}.
\end{align*}
Here we have $g(E_{k,l})=0$ and $K_M\cdot E_{k,l}=0$ since $M$ is a \hk\ manifold and $K_M$ is trivial. Hence $E_{k,l}^2=-2$ for all $k,l$. Thus every component of the exceptional divisor has self-intersection $-2$, and no component over any $b_k$ can be further contracted. Therefore $\chi:M\to S$ is a minimal resolution. This proves the Lemma~\ref{mr}.
\end{proof}

Finally, recall that for a direction $v=(v_1,v_2,v_3)\in \ms^2\subset \R^3$ we may choose complex structure 
\begin{align*}
    J_v=v_1J_x+v_2J_y+v_3J_z
\end{align*}
so that the $\R^3$-projections of the $\C^*$-orbits are parallel to $v$. The projection $\Pi_v$ then plays the role of $H$ in this setting. For $v\in\ms^2$, recall $a_j(v)=\Pi_v(p_j)\in\C$ and denote the distinct values of $\{a_j(v)\}_j$ by $\{b_k(v)\}_k$ with multiplicities
\begin{align*}
  &m_0(v):=\#\{j\in\mathbb{N}^+:a_j(v)=0\},\\
  &m_k(v):=\#\{j\in\mathbb{N^+}:a_j(v)=b_k(v)\}.
\end{align*}

\begin{proof}
    [Proof of Theorem~\ref{vmr}]
    Under the hypothesis of Theorem~\ref{vmr}, $\{a_j(v)\}_j$ have no accumulation point in $\C$, and each multiplicity $m_k(v)$ is finite. We can define the corresponding entire function as
\begin{align*}
  P_v(u):=u^{m_0(v)}\prod_{k:b_k(v)\ne0}\lc E_{\ell_k(v)}\lc\frac{u}{b_k(v)}\rc\rc^{m_k(v)}
\end{align*}
where $\ell_k(v)$ are properly chosen so that the product converges uniformly on every compact subset of $\C$. For such a direction $v$, the above argument in Lemma~\ref{mr} applies verbatim to $(M,J_v)$ with $H$ replaced by $\Pi_v$ and $J$ replaced by $J_v$. Thus $(M,J_v)$ is biholomorphic to the minimal resolution of the hypersurface in $\C^3$ defined by 
\begin{align*}
  u_1u_2=P_v(u_3).
\end{align*}
The singularity type over $u_3=b_k$ is of type $A_{m_k(v)-1}$ for each $k$ with $m_k(v)\ge2$. This proves Theorem~\ref{vmr}.
\end{proof}

\begin{rem}
  Hattori's related work \cite{hattori2014holomorphic} studied \hk\ manifolds of type $A_\wq$ via the \hk\ quotient description and obtains a closely related classification. Our result is parallel to his in some sense.
\end{rem}

\section*{Acknowledgements} 
B.X. would like to express his sincere gratitude to Kota Hattori for his valuable and stimulating conversations during the course of this project. 
B.X. is supported in part by the Project of Stable Support for Youth Team in Basic Research Field, CAS (Grant No. YSBR-001) and NSFC (Grant No. 12271495).

\bibliographystyle{plain}
\bibliography{references}{}
\smallskip

{
\raggedright
WENXIN HE,  BIN XU \\
SCHOOL OF MATHEMATICAL SCIENCES\\
UNIVERSITY OF SCIENCE AND TECHNOLOGY OF CHINA \\
HEFEI 230026 CHINA\\
\texttt{huwenxian@mail.ustc.edu.cn\qquad bxu@ustc.edu.cn}
}

\end{document}